\documentclass[a4paper,USenglish,cleveref, autoref, thm-restate]{lipics-v2021-f}
\usepackage{mathtools}
\usepackage{xspace}
\usepackage{thmtools}
\usepackage{thm-restate}

\usepackage[final,commandnameprefix=ifneeded]{changes}

\pdfoutput=1 
\hideLIPIcs  

\def\final{1}  
\def\iflong{\iffalse}
\ifnum\final=0  

\newcommand{\fvl}[1]{{\color{teal}[{ \textbf{FvL:} #1}]}} 
\newcommand{\ag}[1]{{\color{purple}[{ \textbf{AG:}  #1}]}} 
\newcommand{\ba}[1]{{\color{blue}[{\textbf{BA:} #1}]}} 
\newcommand{\mdb}[1]{{\color{orange}[{ \textbf{MdB:} #1}]}}  
\else 
\newcommand{\aanote}[1]{}
\newcommand{\fvl}[1]{}
\newcommand{\ag}[1]{}
\newcommand{\ba}[1]{}
\newcommand{\mdb}[1]{}
\fi

\newcommand{\etal}{\emph{et al.}\xspace}

\newcommand{\A}{\ensuremath{\mathcal{A}}}

\newcommand{\C}{\ensuremath{\mathcal{C}}}
\newcommand{\D}{\ensuremath{\mathcal{D}}}

\newcommand{\G}{\ensuremath{\mathcal{G}}}

\newcommand{\I}{\ensuremath{\mathcal{I}}}

\newenvironment{myquote}{\list{}{\leftmargin=5mm \rightmargin=5mm}\item[]}{\endlist}

\newcommand{\REAL}{\ensuremath{\mathbb{R}}}
\newcommand{\Reals}{\REAL}

\renewcommand{\leq}{\leqslant}
\renewcommand{\geq}{\geqslant}
\newcommand{\bd}{\partial}

\DeclareMathOperator{\diam}{diam}

\DeclareMathOperator{\ply}{ply}

\newcommand{\graph}{\G}

\newcommand{\dg}{\graph^*}

\newcommand{\mf}{\mu}

\newcommand{\myin}{\mathrm{in}}
\newcommand{\myout}{\mathrm{out}}

\newcommand{\pa}{\mathit{parent}}

\newcommand{\bydef}{\coloneq}

\graphicspath{{./Figures/}}

\title{On the Diameter of Arrangements of Topological Disks}

\author{Aida {Abiad}}{Department of Mathematics and Computer Science,  TU Eindhoven, the Netherlands \and  Department of Mathematics and Data Science, Vrije Universiteit Brussel, Belgium}{a.abiad.monge@tue.nl}{https://orcid.org/0000-0003-4003-4291}{Supported by the Dutch Research Council (NWO) through the grants VI.Vidi.213.085 and OCENW.KLEIN.475.}

\author{Boris Aronov}{Department of Computer Science and Engineering, Tandon School of Engineering, New York University, Brooklyn, NY, USA}{boris.aronov@nyu.edu}{http://orcid.org/0000-0003-3110-4702}{Partially supported by NSF Grant CCF-20-08551.}

\author{Mark de Berg}{Department of Mathematics and Computer Science, TU Eindhoven, the Netherlands}{M.T.d.Berg@tue.nl}{https://orcid.org/0000-0001-5770-3784}{Supported by the  Dutch Research Council (NWO) through Gravitation-grant NETWORKS-024.002.003.}

\author{Julian Golak}{Institute of Sustainable Logistics and Mobility and Institute of Operations Management, University of Hamburg, Hamburg, Germany}{julian.golak@uni-hamburg.de}{https://orcid.org/0000-0003-2626-5819}{}

\author{Alexander Grigoriev}{Department of Data Analytics and Digitalisation, Maastricht University, the Netherlands}{a.grigoriev@maastrichtuniversity.nl}{https://orcid.org/0000-0002-8391-235X}{}

\author{Freija {van Lent}}{Department of Data Analytics and Digitalisation, Maastricht University, the Netherlands}{f.vanlent@maastrichtuniversity.nl}{https://orcid.org/0000-0002-6417-9199}{}

\authorrunning{A. Abiad, B. Aronov, M. de Berg, J. Golak, A. Grigoriev, and F. van Lent} 

\Copyright{Aida Abiad, Boris Aronov, Mark de Berg, Julian Golak, Alexander Grigoriev, and Freija van Lent} 

\ccsdesc[500]{Theory of computation~Computational geometry}
\ccsdesc[300]{Mathematics of computing~Graphs and surfaces}
\ccsdesc[300]{Mathematics of computing~Geometric topology} 

\keywords{Topological disk, arrangement, face, Jordan curve, crossing} 

\relatedversion{} 

\nolinenumbers 

\begin{document}

\maketitle

\begin{abstract}
Let $\D=\{D_0,\ldots,D_{n-1}\}$ be a set of $n$ topological disks in the plane
and let $\A := \A(\D)$ be the arrangement induced by~$\D$. 
For two disks $D_i,D_j\in\D$, let $\Delta_{ij}$ be the number of connected components of~$D_i\cap D_j$, and let $\Delta := \max_{i,j} \Delta_{ij}$. 
We show that the diameter of $\dg$, the dual graph of~$\A$, can be bounded as a function of $n$ and $\Delta$. 
Thus, any two points in the plane can be connected by a Jordan curve that crosses the disk boundaries a number of times bounded by a function of~$n$ and~$\Delta$.
In particular, for the case of two disks, we prove that the diameter of $\dg$ is at most $\max\{2,2\Delta\}$ and this bound is tight. 
For the general case of $n>2$ disks, we show that the diameter of $\dg$ is $O(n^3 2^n \Delta)$.
We achieve this by proving that the number of maximal faces in $\A$---faces whose ply is more than the ply of their neighboring faces---is $O(n^2 2^n \Delta)$. 
To this end, we first show that the number of maximum faces---faces whose ply is~$n$---is $O(n^2\Delta)$;
the latter bound, which is of independent interest, is tight in the worst case.
\end{abstract}

\section{Introduction}
\label{sec:introduction}
The \emph{arrangement} $\A(S)$ defined by a set~$S$ of objects in the plane is the subdivision of the plane into faces, edges, and vertices induced by (the boundaries of) the objects.
Arrangements in the plane and in higher dimensions play a fundamental role in computational geometry, since many geometric problems can be transformed into problems on arrangements. 
As a result, the \emph{combinatorial complexity} of arrangements---the total number of vertices, edges, and faces (and higher-dimensional features)---induced by various types of objects has been studied extensively. 
These studies concern the complexity of complete arrangements as well as the complexity of various substructures: single cells, many cells, zones, levels, and so on; see the book by Sharir and Agarwal~\cite{DBLP:books/daglib/0080837} or the surveys by Agarwal and Sharir~\cite{DBLP:books/el/00/AgarwalS00a} or Halperin and Sharir~\cite{DBLP:reference/cg/arrangements}.  
The known results on arrangements typically deal with constant-complexity objects such as lines, circles, or other objects whose boundaries intersect a bounded number of times. This is not surprising, as the complexity of the arrangement induced by some set of objects obviously cannot be bounded if we do not put a bound on the number of intersections between the boundaries of any two objects in the set. 
In this paper we are interested in whether anything meaningful can be said about complexity of an arrangement if the number of boundary intersections of any two objects in the arrangement is not necessarily bounded. 
\medskip

Let $\D=\{D_0,\ldots,D_{n-1}\}$ be a set of $n$ topological disks in $\Reals^2$. 
To simplify the presentation we will refer to the objects in $\D$ as \emph{disks}, but the reader should keep on mind that the boundaries of the disks in $\D$ are arbitrary closed Jordan curves. 
In particular, the boundaries $\bd D_i$ and $\bd D_j$ of two disks $D_i,D_j\in\D$ can intersect arbitrarily many times.
We will assume that the disks in $\D$ are in \emph{general position}. More precisely, we assume that $\bd D_i\cap \bd D_j$ consists of finitely many points, that two boundaries properly cross every time they intersect, i.e., there are no tangencies, and that no three disk boundaries have a point in common.
\medskip

Let $\A:=\A(\D)$ be the arrangement induced by~$\D$.  
Let $\Delta_{ij}$ denote the number of connected components of the intersection $D_i\cap D_j$ of a pair of objects~$D_i,D_j\in \D$, and define $\Delta := \max_{i,j}\Delta_{ij}$. 
We refer to $\Delta$ as the \emph{overlap number} of~$\D$.
Note that the complexity of $\A$ cannot be bounded as a function of $n$ and~$\Delta$. 
Indeed, even for $n=2$ and $\Delta=1$, the number of faces in $\A$ can be arbitrarily large, as illustrated in Figure~\ref{fig:terminology}(i).
\begin{figure}
\begin{center}
\includegraphics{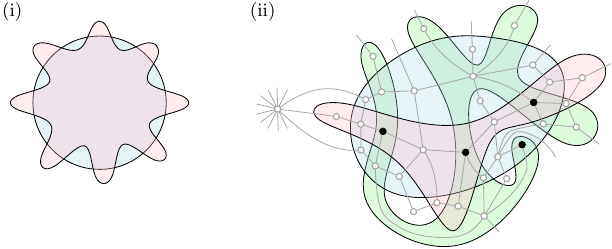}
\end{center}
\caption{(i) A set of two topological disks with $\Delta=1$ can induce an arrangement with arbitrarily many faces.
         (ii) A set of three topological disks and the dual graph of their arrangement.  The arrangement
         has four maximal faces, whose dual nodes are shown in black.
         The edges incident to the node corresponding to the unbounded face are
         only drawn partially, to avoid cluttering the figure. 
         If the blue disk is $D_0$, the red disk is $D_1$, and the green disk is $D_2$,
         then $\Delta_{01}=1$ and $\Delta_{02}=\Delta_{12}=3$, so $\Delta=3$.}
\label{fig:terminology}
\end{figure}
But are there perhaps other quantities describing how complex the arrangement~$\A$ can be, yet be bounded in terms of $n$ and $\Delta$? 
\medskip

The first quantity we will study is the diameter of the dual graph of the arrangement. 
The \emph{dual graph} of the arrangement $\A$ is the graph~$\dg:= \dg(\A)$ whose nodes correspond to the faces in~$\A$ and two nodes in $\dg$ are adjacent if and only if the corresponding faces have at least one edge
in common; see Fig.~\ref{fig:terminology}(ii). 
The \emph{diameter} of $\dg$ is the maximum hop-distance between any two nodes in~$\dg$.
\medskip
 
The second quantity of interest is the number of \emph{maximal faces} in~$\A$. 
Let \emph{ply} of a face in $\A$ be the number of disks in $\D$ that contain the face.
Two faces are called \emph{neighbors} in $\A$ if they have at least one common edge,
if their corresponding nodes in $\dg$ are adjacent.
A face is called \emph{maximal} if its ply is larger than the ply of any of its neighboring faces.
A special case of a maximal face is a \emph{maximum} face, which is a face whose ply is~$n$
or, in other words, a face that is contained in all disks in~$\D$.
\medskip

\subparagraph{Our results.}
Let $\diam(\G)$ denote the diameter of a graph $\G$, 
and let $\mf(\A)$ and $M(\A)$ denote the number of maximal and maximum
faces, respectively, of the arrangement~$\A$.
Furthermore, define
\begin{align*}
\xi(n,\Delta) &\bydef \max_\D \{ \diam(\dg(\A(\D)) \}
\shortintertext{and}
\mu(n,\Delta) &\bydef \max_\D \{ \mf(\A(\D))\}
\shortintertext{and}
M(n,\Delta) &\bydef \max_\D \{ M(\A(\D))\},
\end{align*}
with the maximum taken over all sets $\D$ of $n$ topological disks with overlap number~$\Delta$ in general position.
We first consider the case of two disks. Bounding $\mf(\A)$ is trivial for $n=2$, since $\mf(\A)=\Delta$ 
if $D_0$ and $D_1$ intersect, and $\mf(\A)=2$ if $D_0$ and $D_1$ are disjoint. Thus, $\mu(2,\Delta)=\max\{2,\Delta\}$
and $M(2,\Delta) = \Delta$.
Bounding $\xi(2,\Delta)$ is less straightforward. We study this problem in Section~\ref{sec:two-disks},
where we prove the following tight bound on the maximum possible diameter of~$\dg(\A(\D))$ in terms of~$\Delta$.
\begin{restatable}{theorem}{twodisks} \label{thm:2disks} 
$\xi(2,\Delta) = \max\{2,2\Delta\}$.
\end{restatable}
In Section~\ref{sec:general} we then turn our attention to the much more 
challenging case of general~$n$. Here, we first prove the following
upper bound on the number of maximum faces.
\begin{restatable}{theorem}{maximumFaces} \label{thm:maximum-faces} 
For any $n>2$, we have $M(n,\Delta) < 2n(n-1)\Delta$.
\end{restatable}
We also give a construction showing that $M(n,\Delta) =\Omega(n^2\Delta)$,
showing that the bound from \autoref{thm:maximum-faces} is asymptotically tight.
This result implies the following bound on the number of maximal faces.
\begin{restatable}{theorem}{mfgeneral} \label{thm:mf-general} 
For any $n>2$, we have $\mf(n,\Delta) =O(n^2 2^n \Delta)$.
\end{restatable}
Finally, we show that the diameter of $\dg$ is bounded in terms of ~$n$ and $\Delta$.
\begin{restatable}{theorem}{diamgeneral} \label{thm:diam-general} 
For any $n>2$, we have $\xi(n,\Delta) =O(n^3 2^n \Delta)$.
\end{restatable}
\subparagraph{Related work.}
The problem of bounding the diameter of $\dg$ is equivalent to the following question.
For two points $s,t\in \Reals^2$ not lying on the boundary of any disk in $\D$, let $d(s,t)$ be the minimum number of times one needs to cross a disk boundary when going from $s$ to~$t$. 
Then, the question is: what is the maximum value of $d(s,t)$ over all such pairs~$s$ and $ t$?
This is related to the concept of \emph{barrier coverage} in sensor networks, which aims to ensure that all possible paths through a surveillance domain will intersect the boundary of a sensor region at least once.  
Cardei and Wu~\cite{cardei2004coverage} and Meguerdichian~\etal~\cite{meguerdichian2001coverage}  provide comprehensive overviews of work on this topic. 
The notion of barrier coverage closest to our work is the one introduced by Bereg and Kirkpatrick \cite{bereg2009approximating}, who define the \emph{thickness} of a barrier as the minimum number of sensor region intersections of a path between any two regions. 
A substantial difference between the paper by Bereg and Kirkpatrick and our paper is the way of counting the number of crossings of an $st$-path  with the region boundaries:
they count a boundary crossing only if the curve enters a region for the first time, 
while we count every boundary crossing. 
Cabello~\etal~\cite{DBLP:journals/ijcga/AltCGK17} study the related problem of
computing, for a given set of line segments and two points $s$ and $t$, the minimum
number of segments crossed by an $st$-path, but they also count each segment at most once. 

\subparagraph{Notation.}
In the following, $\D=\{D_0,\ldots,D_{n-1}\}$ always refers to a set of $n$ topological disks 
in general position. Furthermore, $\A$ refers to the arrangement induced by~$\D$, 
and $\dg$ to the dual graph of~$\A$. 
For two disks $D_i,D_j\in \D$, $\C(D_i,D_j)$ denotes the set of connected components of $D_i\cap D_j$. 
Recall that $\Delta_{ij} := |\C(D_i,D_j)|$ denotes the number of such connected components, and that the \emph{overlap number} of $\D$ is defined as $\Delta := \max_{ij}\Delta_{ij}$.
For two nodes $u$ and $v$ in the dual graph $\dg$ of the arrangement~$\A$, we use $d(u,v)$ to denote the distance between $u$ and $v$ in~$\dg$, that is, the minimum number of edges on any path from $u$ to $v$ in~$\dg$. 
When we speak of the \emph{hop-distance} between two faces $f,f'$ in the arrangement~$\A$, we refer to the distance between the nodes in $\dg$ that correspond to $f$ and~$f'$. 
With a slight abuse of notation, we denote this hop-distance by $d(f,f')$.


\section{The case of two disks}
\label{sec:two-disks}
\fvl{ISAAC: Add more illustrations for the n > 2 case}
In this section we will prove Theorem~\ref{thm:2disks}. 
We first prove the upper bound --- that is, we show that the diameter of the dual graph~$\dg$ of an arrangement of two topological disks is at most $\max\{2,2\Delta\}$ --- and then, we provide an example showing that this bound is tight.

\subparagraph{The upper bound.}
It is easy to verify that $\diam(\dg)=2$ when $\Delta=0$ or $\Delta=1$, so from now on we assume that $\Delta\geq 2$. 
Let $D_0$ and $D_1$ be the two disks, and imagine that $D_0$ is colored red and $D_1$ is colored blue.
We color the faces of $\A(\D)$ accordingly: faces contained in $D_0\setminus D_1$ are red,  faces in $D_1\setminus D_0$ are blue, faces in $D_0\cap D_1$ are purple, and faces in $\Reals^2 \setminus (D_0\cup D_1)$ are white; see Fig.~\ref{fig:dual-of-two-disks}.
Similarly, we color the nodes of $\dg$ in the colors of their respective faces.
Observe that any edge in $\dg$ is either between a white node and a red or blue node, or between a purple node and a red or blue node; there are neither edges between white and purple nodes, nor edges between red and blue nodes.
Also observe that the number of purple nodes is, by definition, exactly~$\Delta$.
We call a path in $\dg$ a \emph{colored path} if it only uses non-white nodes.
\begin{lemma}\label{lem:purple-to-purple}
For any two purple 
nodes~$u,v$ in $\dg$ there is a colored path from $u$ to $v$ of length at most $2\Delta-2$.    
\end{lemma}
\begin{proof}
First, notice that a colored path from $u$ to $v$ always exists since $u$ and $v$ are purple. 
Indeed, the corresponding faces~$f_u$ and $f_v$ lie in $D_0\cap D_1$, so there is a curve $\gamma\subset D_0$ connecting $f_u$ and $f_v$ inside $D_0$. This curve corresponds to a colored path in~$\dg$.

Let $\pi(u,v)$ be a shortest colored path from $u$ to $v$  in~$\dg$. 
Since $\pi(u,v)$ is a shortest path, any node occurs at most once on it.
Moreover, $\pi(u,v)$ does not have white nodes, so it alternates between purple nodes and red or blue nodes. 
Since $\pi(u,v)$ starts and ends at a purple node and it contains at most $\Delta$ purple nodes, $\pi(u,v)$ has at most $2\Delta-2$ edges.
\end{proof}
\begin{figure}
\begin{center}
\includegraphics{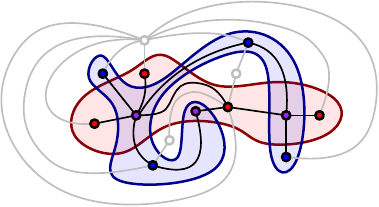}
\end{center}
\caption{The dual graph of an arrangement of two disks, with white, blue, red, and purple nodes.}
\label{fig:dual-of-two-disks}
\end{figure}
Since we assumed $\Delta > 0$, any red or blue node has a purple neighbor. 
By Lemma~\ref{lem:purple-to-purple} the distance between any two purple nodes is at most $2\Delta-2$ that implies $d(u,v)\leq 2\Delta$ for any two non-white nodes~$u$ and $v$.
Thus, to prove that $\diam(\dg)\leq 2\Delta$, it remains to consider the case where at least one of the nodes~$u,v$ is white. 
Observe that any white node has a red neighbor (and a blue neighbor as well), and so, it can reach a purple node in two steps. 
This implies that $d(u,v)\leq 2\Delta+2$ for any two nodes~$u,v$: we can reach a purple node $w_1$ from $u$ in at most two steps, we can reach a purple node $w_2$ from $v$ in at most two steps, and $d(w_1,w_2)\leq 2\Delta-2$ by Lemma~\ref{lem:purple-to-purple}. 
Next we show that $d(u,v)\leq 2\Delta$ even if one or both of the nodes~$u,v$ are white. 
The improvement is based on the following lemma.
\begin{lemma} \label{lem:white-to-purple}
Let $\Delta\geq 2$ and let $u$ be a white node in $\dg$. 
Then, there are two distinct purple nodes $w,w'$ with $d(u,w)=d(u,w')=2$. 
\end{lemma}\fvl{ISAAC: The unbounded case does not need to be separate. You can simply consider an alternative embedding where that face is not the unbounded face. 
		Think of placing the arrangement on one side of a sphere, and then expanding out an arbitrary face so that the rest is now on the opposite side.
}
\begin{proof}
Let $u$ be a white node in~$\dg$ and $f_u$ its corresponding face in~$\A$.
There are two cases to consider.
\\[2mm]
\emph{Case~I: $f_u$ is a bounded face.} 
Let $v_1,\ldots,v_{t}$ be the vertices of $\A$ incident to $f_u$, enumerated in order along $\bd f_u$.
Suppose for a contradiction that all purple faces incident to the vertices of~$f_u$ are one and the same purple face $f$. Note that $d(f,f_u)=2$. 
Let $x$ be a point in the interior of~$f$. Draw for each vertex~$v_i$ a curve $\gamma_i \subset f$ from $x$ to~$v_i$; see Fig.~\ref{fig:connect}(i).
It will be convenient to draw the curves such that they are pairwise disjoint except at the shared endpoint~$x$.
\begin{figure}
\begin{center}
\includegraphics{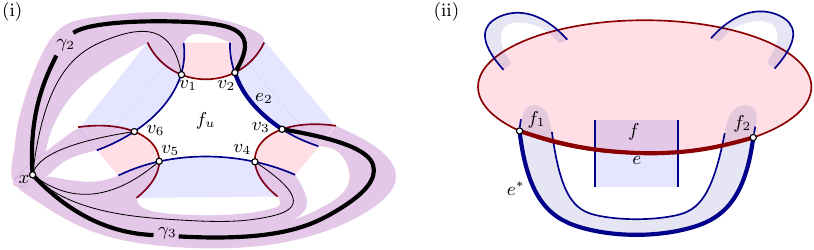}
\end{center}
\caption{(i) The edge $e_2=v_2 v_3$ and the curves $\gamma_2$ and $\gamma_3$ form a closed curve
enclosing $f_u$. (ii) If the purple faces $f_1$ and $f_2$ are the same face, 
there is a curve $\gamma\subset f_1$ connecting the endpoints of $e^*$.}
\label{fig:connect}
\end{figure}
Then, there must be an edge $e_i=v_i v_{i+1}$, with indices taken modulo~$t$, that, together with the two curves $\gamma_i,\gamma_{i+1}$, forms a closed curve $\Gamma$ enclosing~$f_u$.
To see this, imagine adding the curves $\gamma_1,\ldots,\gamma_t$ one by one and consider the arrangement defined by $\bd f$ (which is a simple closed curve) and the added curves~$\gamma_i$. 
It is easy to argue by induction that the unbounded face of this arrangement contains an edge $e\subset \bd f_u$. 
This edge is the edge $e_i$ with the required properties.
If $e_i\subset \bd D_0$ then $\Gamma \subset D_0$, and if $e_i\subset \bd D_1$ then $\Gamma \subset D_1$, both contradicting that $f_u$ is a white face.
\\[2mm] 
\emph{Case~II: $f_u$ is the unbounded face.} 
In this case, we can assume there is a purple face $f$ with $d(f,f_u)>2$, otherwise we are done. 
Pick any edge $e\subset \bd D_0$ of $f$.  
Note that $e$ is not an edge of $f_u$, since $f$ is purple and $f_u$ is white.  
Now consider $\left( \Reals^2 \setminus f_u \right) \setminus D_0$.
This region consists of several components, one for each edge of $f_u$ contributed by~$\bd D_1$. 
Note that $e$ is contained in the boundary of one of these components, which we denote by $C$.
Let $e^* := \bd C \cap \bd D_1$
; see Fig.~\ref{fig:connect}(ii).
Let $f_1$ and $f_2$ be the two purple faces incident to the endpoints of $e^*$.
Then $f_1$ and $f_2$ must be distinct purple faces, thus proving the claim. 
Indeed, if $f_1=f_2$ then there is a curve $\gamma\subset f_1$ connecting the endpoints of $e^*$ such that $\gamma\cup e^*$ is a closed curve contained in $D_1$ and enclosing~$e$, contradicting that $f\neq f_1,f_2$, while we must have $f\neq f_1,f_2$ since $d(f,f_u)>2$.
\end{proof}
With Lemmas~\ref{lem:purple-to-purple} and~\ref{lem:white-to-purple} available, we now argue that $d(u,v)\leq 2\Delta$ for any two nodes $u,v$ in~$\dg$.
We already proved this when both $u$ and $v$ are non-white. 
Now assume that $u$ is a white node. 
Let $d_{\max}$ be the maximum distance between any two purple nodes. 
We have two cases.
\begin{itemize}
\item If $d_{\max}<2\Delta-2$, then $d_{\max}\leq 2\Delta-4$. Indeed, any path in $\dg$ alternates between a purple or white node and a blue or red node. So, any path between two purple nodes has even length. Since both $u$ and $v$ can reach a purple node in at most two steps, we have $d(u,v)\leq d_{\max}+2+2\leq 2\Delta$. 
\item If $d_{\max}=2\Delta-2$, by Lemma~\ref{lem:purple-to-purple} there are two purple nodes $w_1,w_2$ whose shortest colored path $\pi(w_1,w_2)$ has length~$2\Delta-2$ and passes through all the purple nodes. By Lemma~\ref{lem:white-to-purple}, in two steps $u$ can reach at least two distinct purple nodes on this path.  This implies that it can reach every purple node in at most $2\Delta-2$ steps. Since $v$ can reach a purple node in at most two steps, $d(u,v)\leq (2\Delta-2)+2=2\Delta$.
%
\end{itemize}
This finishes the proof of the upper bound of Theorem~\ref{thm:2disks}.

\subparagraph{The lower bound.}
\begin{figure}
    \centering
    \includegraphics[width=\linewidth]{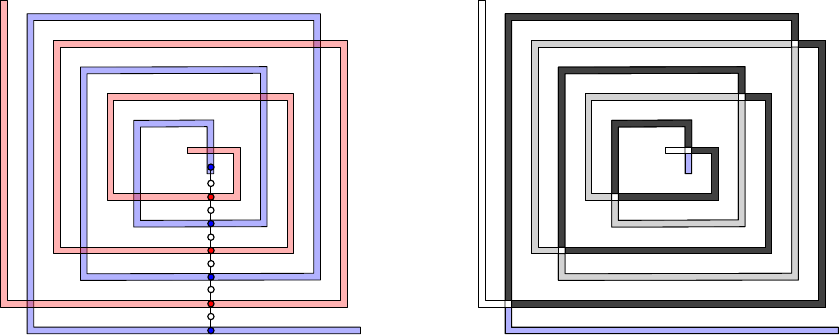}
    \caption{Left: Two spiral-shaped topological disks with overlap number~$\Delta=6$.
             The path shown in the figure is a shortest path in the dual graph between 
             the nodes corresponding to the innermost blue face and the outermost blue face. 
             Its length is~$2\Delta=12$.  
             Right: Any path between the innermost and the outermost blue face must cross each
             of the black and gray rings.}
    \label{fig:lower-bound-two-disks}
\end{figure}
The dual graph $\dg$ of two disjoint disks has diameter~2.
Therefore, to prove the lower bound, it suffices to construct an arrangement of two disks with the overlap number $\Delta\geq 2$ whose dual graph has diameter~$2\Delta$.
Fig.~\ref{fig:lower-bound-two-disks} exhibits such an arrangement for $\Delta=6$. 
Each spiral consists of $4k$ thin corridors, two vertical and two horizontal.
In the figure $k=3$, but the example easily extends to larger~$k$. 
Note that when increasing $k$ by~1, we increase $\Delta$ by~2 as each vertical corridor of the blue disk intersects a horizontal corridor of the red disk.
Thus, in general, $k=\Delta/2$. In the arrangement, we can identify $2k-1=\Delta-1$ rings, each consisting of a red and a blue face at equal distance from the innermost blue face; see the right image in Fig.~\ref{fig:lower-bound-two-disks}, where the rings are shown in black and gray. 
Going from the innermost blue face to the innermost ring requires two hops, going from a ring to the next ring requires two hops, and going from the outermost ring to the outermost blue face requires two hops.
Thus, the distance between the innermost and outermost blue face is~$2\Delta$.

This construction gives an arrangement of two disks with overlap number $\Delta>1$ whose dual graph has diameter~$2\Delta$, for even $\Delta\geq 2$.
By omitting the last two corridors of both disks, we obtain such an arrangement for odd~$\Delta>2$ as well.

\section{The general case of \texorpdfstring{$n$}{n} disks}
\label{sec:general}
We now consider the case of general $n$, which turns out to be significantly harder 
than the case of two disks. The main difficulty lies in bounding the number
of maximal faces; for two disks this number is simply the overlap
number $\Delta$ (assuming the disks are not disjoint), but for general $n$ it
is not even clear that the number of maximal faces {(or the number of maximum faces,
for that matter)} is bounded. Our main effort
lies in proving that this is the case (Section \ref{sec:bound-max-faces}). Once we have shown that, establishing
a bound on the diameter of the dual graph of the arrangement of disks is relatively easy (Section \ref{sec:bounddiamdualgraph}).

\subsection{The proof of Theorems~\ref{thm:maximum-faces} and~\ref{thm:mf-general}} \label{sec:bound-max-faces}
Recall that $\mf(A)$ and $M(\A)$ denote the number of maximal and maximum faces,
respectively, in the arrangement~$\A$ defined by the set $\D$ of $n$ disks.  
We first prove our bound on the number of maximum faces.
\maximumFaces*
\begin{proof}
Clearly, $M(2,\Delta)=\Delta$ by definition. We will show that for $n\geq 3$ we have
\begin{equation} \label{eq:max}
M(n,\Delta) \leq M(n-1,\Delta) + 4(n-1)(\Delta-1).
\end{equation}
This implies the theorem, since
\[
\begin{array}{lll}
 M(n,\Delta) &  \leq & M(2,\Delta) + \sum_{j=3}^{n} 4(j-1)(\Delta-1) \\
     & = &  \Delta + \sum\limits_{j=3}^{n} 4(j-1)(\Delta-1) \\
     & < & \sum\limits_{j=2}^{n} 4(j-1)\Delta  \mbox{ \mdb{We could be more precise, but I prefer the shorter formula.}}  \\
     & = &  2n(n-1)\Delta. 
\end{array}
\]
To prove that Inequality~(\ref{eq:max}) holds, 
consider a set $\D$ of $n\geq 3$ disks with overlap number~$\Delta$, and let $D_0$ be an arbitrary
disk in~$\D$. Define $\D_1 := \D\setminus \{D_0\}$ and let $D_1,\ldots,D_{n-1}$ be the disks in~$\D_1$.
We will show that
\[
M(\A(\D)) \leq M(\A(\D_1)) + 4(n-1)(\Delta-1).
\]
Note that some faces do not change when we remove $D_0$ from $\D$, 
while other faces---the ones incident to $\bd D_0$---expand into larger faces. Now consider a 
maximum face $f$ of $\A(\D)$ and let $\pa(f)$ be the face in $\A(\D_1)$ that contains~$f$. 
Note that $\pa(f)$ is a maximum face
in~$\A(\D_1)$. We consider two types of maximum faces in~$\A(\D)$, 
illustrated in \autoref{fig:maxfaces}(i).
\begin{figure}
\begin{center}
\includegraphics{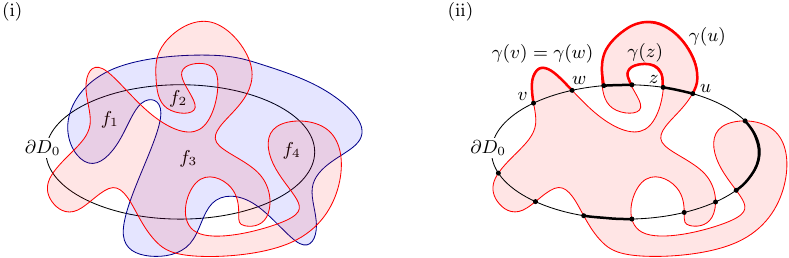}
\end{center}
\caption{(i) An arrangement of three disks. Faces $f_1$ and $f_4$ are safe, faces $f_2$ and $f_3$ are unsafe.
          (ii) The red disk defines four dangerous intervals along $\bd D_0$, indicated by thick lines.}
\label{fig:maxfaces}
\end{figure}
\begin{itemize}
\item A \emph{safe face} is a maximum face $f$ in $\A(\D)$ such that $f$ is the unique maximum
      face in $\A(\D)$ contained in $\pa(f)$.
\item An \emph{unsafe face} is a maximum face $f$ in $\A(\D)$ such that there is a maximum
      face $f'\neq f$ in $\A(\D)$ that is also contained in $\pa(f)$.
\end{itemize}
The number of safe faces is obviously at most $M(\D_1)$, so  
\[
\begin{array}{lll}
M(\D) & = & (\mbox{number of safe faces}) + (\mbox{number of unsafe faces}) \\
     & \leq & M(\D_1) + (\mbox{number of unsafe faces}).
\end{array}
\]
Note that the inequality is strict, unless there are no unsafe faces. Indeed, any
set of unsafe faces with the same parent face already contributes~1 to $M(\D_1)$.
Next, we bound the number of unsafe faces. 

Consider the arrangement $\A(\{D_0,D_i\})$ induced by $D_0$ and~$D_i$.
Any vertex~$v$ of this arrangement is incident to two edges of $\A(\{D_0,D_i\})$
contained in~$\bd D_i$, one inside $D_0$ and one outside $D_0$. We call the latter edge
the \emph{outside curve} of~$v$ and we denote it by~$\gamma(v)$.
Let $\I(D_i) := D_i \cap \bd D_0$ be the (possibly empty) set of intervals on $\bd D_0$ covered by~$D_i$.
Consider an interval $I\in \I(D_i)$ and let $v$ and $w$ be its endpoints. 
We call the interval~$I$ \emph{dangerous} if $\gamma(v)\neq \gamma(w)$;
see \autoref{fig:maxfaces}(ii) for an example.
\begin{claim} \label{claim:dangerous}
Any disk $D_i\in\D_1$ defines at most $2\Delta-2$ dangerous intervals.
\end{claim}
\begin{claimproof}
Assume there is at least one dangerous interval, otherwise the claim holds trivially. 
Let $F_i$ be the set of faces of $\A(\{D_0, D_i\})$ that are incident
to a dangerous interval, let $F_i^{\myin} \subset F_i$ be the faces
that lie inside~$D_0$ and let $F_i^{\myout} \subset F_i$ be the faces
that lie outside~$D_0$. Let $\G$ be the graph whose nodes are the faces in $F_i$ 
and where there is an edge between two faces if they share a dangerous interval;
see \autoref{fig:dual-of-Ai}.
\begin{figure}
\begin{center}
\includegraphics{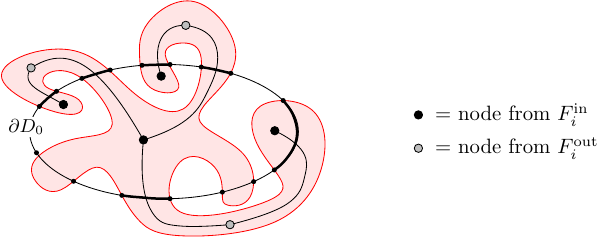}
\end{center}
\caption{The graph~$\G$ used in the proof of \autoref{claim:dangerous}.}
\label{fig:dual-of-Ai}
\end{figure}
$\G$ is a bipartite graph with node sets $F_i^{\myin}$ and $F_i^{\myout}$,
and that $|F_i^{\myin}|\leq \Delta$ by definition of the overlap number.
Moreover, any node~$f\in F_i^{\myout}$ has degree at least two in~$\G$.
Indeed, if $f$ had degree one, then the face~$f$ would be adjacent to a single
dangerous interval, which is impossible by the definition of dangerous intervals.
Note that $\G$ is a tree, since $D_i$ is a topological disk.

We now view $\G$ as a rooted tree, where we pick an arbitrary node~$f\in F_i^{\myout}$ as the root. 
Since the nodes in $F_i^{\myout}$ have degree at least~2 in~$\G$, they are not leaves. 
To bound $|F_i^{\myout}|$, we charge each node in $F_i^{\myout}$ to an arbitrary one of its children.
Thus we charge each node in $F_i^{\myout}$ to a node in~$F_i^{\myin}$ in such a way
that each node in $F_i^{\myin}$ is charged at most once. 
Since the root node has at least two children,
one of which will not be charged, we conclude that
\[
|F_i^{\myout}| \leq |F_i^{\myin}| - 1 \leq \Delta -1.
\]
Hence, the total number of nodes in $\G$ is at most $2\Delta-1$,
and so the number of edges in~$\G$, which equals the number of dangerous intervals,
is at most $2\Delta-2$.
\end{claimproof}
We call an edge of a face $f$ of $\A(\D)$ a \emph{boundary edge} of $f$ if it is contained in~$\bd D_0$.
The next claim relates the number of dangerous intervals to the number of unsafe faces.
\begin{claim}\label{claim:dangerous-to-unsafe}
Let $f$ be an unsafe face. Then $f$ has a boundary edge~$e$ such that 
at least one of the endpoints of~$e$ is an endpoint of a dangerous interval.
\end{claim}
\begin{claimproof}
Let $f'\neq f$ be a maximum face in $\A(\D)$ that is contained in~$\pa(f)$;
such a face exists because $f$ is unsafe.
Then there is a curve $\xi$ in the interior of $\pa(f)$ that connects $f$ to~$f'$.
(More precisely, $\xi$ connects a point in the interior of~$f$ to a point in the interior of~$f'$.)
Since $f$ and $f'$ are distinct faces in~$\A(\D)$, this curve~$\xi$
must cross a boundary edge, $e$, of~$f$. 
We can assume without loss of generality that $\xi$ crosses $e$ once.
Let $v$ and $w$ be the endpoints of~$e$, and assume for
a contradiction that neither $v$ nor $w$ is an endpoint of a dangerous interval.
Let $\gamma(v)$ and $\gamma(w)$ be the outside curves of $v$ and $w$, respectively, 
and let $D_i$ and $D_j$ be such that $\gamma(v)\subset \bd D_i$ and $\gamma(w)\subset \bd D_j$
We consider two cases.
\begin{figure}[b]
\begin{center}
\includegraphics{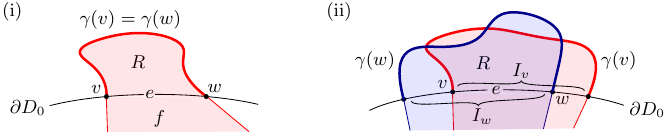}
\end{center}
\caption{Illustration for the proof of \autoref{claim:dangerous-to-unsafe}.}
\label{fig:unsafe-is-dangerous}
\end{figure}
\begin{itemize}
\item If $i=j$ then $\gamma(v)=\gamma(w)$, because we assumed that
      neither $v$ nor $w$ is an endpoint of a dangerous interval.
      Hence, $f'$ is contained in the region $R$ enclosed by $\gamma(v)\cup e$;
      see \autoref{fig:unsafe-is-dangerous}(i).
      But $R$ lies fully outside~$D_0$, which contradicts that $f'$ is a maximum face in~$\A(\D)$.
\item If $i\neq j$ then $\gamma(v)$ must intersect $\gamma(w)$. To see this,
      consider the interval $I_v \subset \bd D_0\cap D_i$ with endpoint $v$ and containing~$e$,
      and the interval $I_w \subset \bd D_0\cap D_j$ with endpoint $w$ and containing~$e$.
      Then $w\in I_v$ and $v\in I_w$. Together with the fact that $I_v$ and $I_w$ are 
      not dangerous by assumption, this implies that $\gamma(v)$ indeed intersects $\gamma(w)$;
      see \autoref{fig:unsafe-is-dangerous}(ii). 
      Let $R$ be the bounded face in the arrangement defined by $e$, $\gamma(v)$
      and $\gamma(w)$ that is incident to $e$. Then $f'$ is contained in $R$, 
      which lies outside $D_0$, contradicting that $f'$ is a maximum face in~$\A(\D)$.
\end{itemize}
We derived a contradiction in both cases, thus finishing the proof of the claim.
\end{claimproof}
\autoref{claim:dangerous-to-unsafe} allows us to charge each unsafe face to a dangerous
interval in such a way that each unsafe interval is charged at most once for each of its endpoints.
By \autoref{claim:dangerous}, the total number of dangerous intervals is at most~$(n-1)(2\Delta-2)$.
Hence, the number of unsafe faces is at most~$2(n-1)(2\Delta-2)$.
This finishes the proof of Inequality~(\ref{eq:max}) and, hence, of the theorem.
\end{proof}
The construction in \autoref{fig:max-faces-lower-bound} shows that $M(n,\Delta) =\Omega(n^2 \Delta)$.
\begin{figure}
\begin{center}
\includegraphics[width=\linewidth]{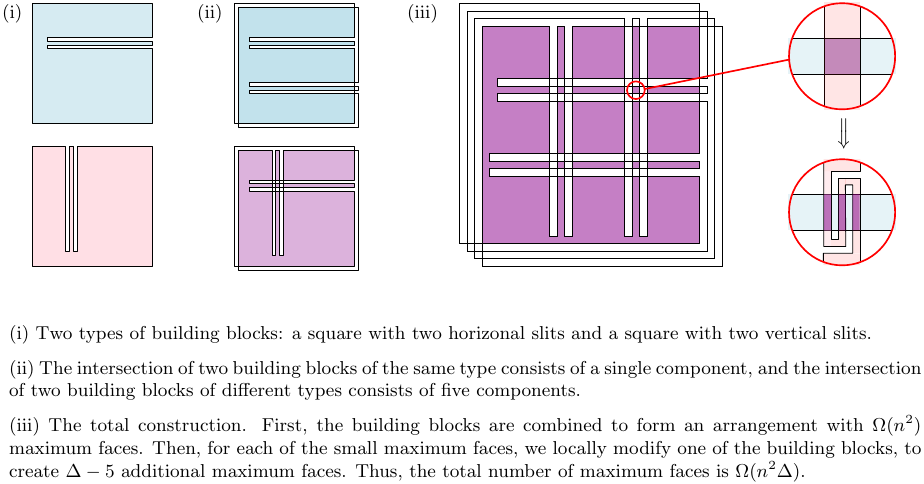}
\end{center}
\caption{A construction showing that $M(n,\Delta) =\Omega(n^2 \Delta)$.}
\label{fig:max-faces-lower-bound}
\end{figure}
\medskip

\autoref{thm:maximum-faces} implies the following bound on the number of maximal faces.

\mfgeneral* 
\begin{proof}
We claim that for any $n\geq 2$ we have $\mu(n,\Delta) \leq (2^n-1) \cdot M(n,\Delta)$.
Together with \autoref{thm:maximum-faces}, this implies the theorem.

To prove the claim, let $\D$ be a collection of $n$ topological disks with overlap number~$\Delta$,
and let $f$ be a maximal face of $\A(\D)$. Then $f$ is a maximum face in the
arrangement $\A(\D_f)$, where $\D_f\subseteq \D$ is the set of disks containing~$f$.
There are at most $M(|\D_f|,\Delta)$ such faces, which we charge to~$D_f$.
Hence,
\[
\mu(n,\Delta) \leq \sum_{\emptyset\neq\D'\subseteq \D} M(|\D'|,\Delta) \leq \sum_{\emptyset\neq\D'\subseteq \D} M(n,\Delta) = (2^n-1)  \cdot M(n,\Delta), 
\]
as claimed. For the second inequality, note that $M(n,\Delta)$ is non-decreasing in~$n$.
\end{proof}

\subsection{The proof of Theorem~\ref{thm:diam-general}}\label{sec:bounddiamdualgraph}
Using the bound on the number of maximal faces in the arrangement~$\A$, we now obtain a bound on the diameter of the dual graph~$\dg$ of~$\A$.
  
Recall that the \emph{ply} of a face $f$ in the arrangement $\A$, denoted by~$\ply(f)$, 
is the total number of disks in $\D$ that contain~$f$. 
Define $p_{\max} := \max_f \ply(f)$ to be the maximum ply over  all faces in the arrangement~$\A$. 
For a node $u$ in the dual graph $\dg$, we define $\ply(u) := \ply(f_u)$, where $f_u$ is the face in $\A$ corresponding to~$u$.
The next lemma relates $\mf(\A)$, the number of maximal faces in~$\A$, to $\diam(\dg)$, the diameter of the dual graph~$\dg$ of~$\A$. 
Combining the lemma with Theorem~\ref{thm:mf-general} proves Theorem~\ref{thm:diam-general}.
\begin{lemma}  \label{lem:max-to-diam}
Let $\A$ be an arrangement induced by a set $\D$ of $n$ disks, and let $\dg$ be the dual graph of~$\A$.
Then $\diam(\dg) \leq 2 p_{\max} \cdot \mf(\A) \leq 2n \cdot \mf(\A)$.
\end{lemma}
\begin{proof}
Let $V$ be the set of nodes of~$\dg$, and let $M=\{v_1,\ldots,v_t\}$ be the subset of nodes from $V$ corresponding to maximal faces in~$\A$. 
We say that a node $u\in V$ has a \emph{monotone path} to a node $v_i\in M$ if there is a path $u=w_1,\ldots,w_s=v_i$ in $\dg$ such that $\ply(w_{j+1})=\ply(w_j)+1$ for all $1\leq j<s$. 
Observe that a node $v_i\in M$ has a monotone path to itself and to no other node~$v_j\in M$.
We partition $V$ into subsets $V_1,\ldots,V_t$, one for each node $v_i\in M$, where we define
\[
V_i := \{ u\in V : \mbox{$u$ has a monotone path to $v_i$ but not to any $v_j$ with $j<i$} \}.
\]
Note that each node $u\in V$ is assigned to a set~$V_i$, because from $u$ we can greedily move to faces of higher ply until we reach a maximal face. 
Also note that $v_i\in V_i$ for all~$i$. 

Now, let $\overline{\dg}$ be the graph obtained from $\dg$ by contracting each 
subgraph $\dg[V_i]$ to a single node. 
We make the following two observations.
\begin{enumerate}[(i)]
\item  $\overline{\dg}$ is a connected graph (because $\dg$ is obviously connected) with $\mf(\A)$ nodes. Hence, the distance between any two nodes $V_i,V_j$ in $\overline{\dg}$ is at most $\mf(\A)-1$.
\item Let $v_i,v_j\in M$ be such that the nodes in $\overline{\dg}$ corresponding to the sets~$V_i$ and $V_j$ are adjacent. Then, the distance $d(v_i,v_j)$ between the nodes $v_i$ and $v_j$ in $\dg$ is at most $\ply(v_i)+\ply(v_j)$. To see this, let $w_i\in V_i$ and $w_j\in V_j$ be adjacent nodes in~$\dg$. Such nodes must exist because $V_i$ and $V_j$ are adjacent in~$\overline{\dg}$. From $w_i$ we can walk to $v_i$ in $\ply(v_i)-\ply(w_i)$ hops and from $w_j$ we can walk to $v_j$ in $\ply(v_j)-\ply(w_j)$ hops. Hence, 
      \[
      d(v_i,v_j) \leq \big( \ply(v_i)-\ply(w_i) \big) + \big( \ply(v_j) - \ply(w_j) \big) + 1 \leq \ply(v_i) + \ply(v_j),
      \]
      where the last inequality follows because $\ply(w_i)$ and $\ply(w_j)$ cannot both be zero.
\end{enumerate}
Now, let $u,u'$ be two arbitrary nodes in~$V$. 
Let $i,j$ be such that $u\in V_i$ and $u'\in V_j$. 
We can walk from $u$ to $v_i$ in $\ply(v_i)-\ply(u)\leq p_{\max}$ hops, and we can walk from $u'$ to $v_j$ in $\ply(v_j)-\ply(u')\leq p_{\max}$ hops.
Moreover, by combining the two observations above, we see that there is a path from $v_i$ to $v_j$ in $\dg$ consisting of $(\mf(\A)-1) \cdot 2 p_{\max}$ edges. 
Hence, $d(u,u') \leq 2 p_{\max} \cdot \mf(\A)$. 
It remains to observe that $p_{\max}\leq n$, which we trivially have.
\end{proof}

\section{Concluding Remarks and Discussion}
In this paper, we derived bounds on the diameter of the dual graph of an arrangement~$\A$ 
of disks in terms of the number of disks, $n$, and the maximum number of intersection 
components between any two disks, $\Delta$. {(The bound on the diameter
is equivalent to a bound on the number of crossings of 
any minimum-crossing $st$-curve in arrangements of two disks.)}

For an arrangement containing only two disks, 
we established a sharp linear bound of \(\max\{2,2\Delta\}\). 
\deleted{Such an upper bound is equivalent to a bound on the number of crossings of 
any minimum-crossing $st$-curve in arrangements of two disks.} 
For an arrangement of $n$ disks, we bounded \deleted{both} the diameter of the dual graph
{by $O(n^3 2^n\Delta)$. This was done by establishing an $O(n^2 2^n\Delta)$ bound on the 
number of maximal faces in~$\A$. The latter was done by first proving an $O(n^2\Delta)$ bound on the number
of maximum faces, which is asymptotically tight.}
\deleted{and the number of maximal faces as a function of $n$ and $\Delta$. The bound on the number of maximal faces $\mf(\A)$ is $n(\Delta+1)^{n(n-1)/2}$. Therefore the upper bound on the diameter of the dual graph of arrangements with $n$ disks is $2 n(\Delta+1)^{n(n-1)/2} \min\{n,\Delta+1\}$, which corresponds to the bound for the number of crossings of any minimum-crossing $st$-curve.} 

Our bound on the number of maximal faces is far from tight: our best lower bound 
is the same as for the number of maximum faces, namely $\Omega(n^2\Delta)$, while our upper bound
is $O(n^2 2^n\Delta)$. The main open problem is to tighten this gap. We believe
that the true bound is polynomial in~$n$---in fact, we conjecture that the $\Omega(n^2\Delta)$ lower bound is asymptotically tight.
The upper bound on the number of maximal faces was used to derive a bound on the diameter
of the dual graph of the arrangement. Notice that by making $n/2$ nested copies of each of the two disks in 
\autoref{fig:lower-bound-two-disks},  we can produce an arrangement of $n$ disks with overlap $\Delta$ and diameter~$\Omega(n \Delta)$.
Instead of trying to improve the $O(n^3 2^n\Delta)$ upper bound by improving the bound on the number
of maximal faces, one may also try to prove a better bound directly.



\bibliography{references}

\appendix

\end{document}